\documentclass[11pt]{article}

\usepackage{amsmath,amssymb,amsthm}
\usepackage{algorithm}
\usepackage{algorithmicx}

\usepackage{graphicx}
\DeclareGraphicsExtensions{.png}
\graphicspath{{.}}

\usepackage{url}

\newcommand{\M}{\ensuremath{{\cal M}}}
\newcommand{\G}{\ensuremath{{\cal G}}}
\newcommand{\HH}{\ensuremath{{\cal H}}}

\newcommand{\dist}{\ensuremath{\:\mbox{\rm dist}}}
\newcommand{\Id}{\ensuremath{\mbox{\rm Id}}}

\newcommand{\inv}{^{\text{\tiny (-1)}}}

\newtheorem{theorem}{Theorem}

\begin{document}
\pagenumbering{arabic}
\pagestyle{headings} 


\title{Parallel Transport with Pole Ladder: a Third Order Scheme in Affine Connection Spaces which is Exact in Affine Symmetric Spaces}

\author{Xavier Pennec} 

\maketitle



\begin{abstract}
Parallel transport is an important step in many discrete algorithms for statistical computing on manifolds. Numerical methods based on Jacobi fields or geodesics parallelograms are currently used in geometric data processing. In this last class,  pole ladder is a simplification of Schild's ladder for the parallel transport along geodesics that was shown to be particularly simple and numerically  stable in Lie groups \cite{lorenzi:hal-00870489}. So far, these methods were shown to be first order approximations of the Riemannian parallel transport, but
higher order error terms are difficult to establish. 

In this paper, we build on a BCH-type formula on affine connection spaces \cite{gavrilov_algebraic_2006} to establish the behavior of one pole ladder step up to order 5. It is remarkable that the scheme is  of order three in general affine connection spaces with a symmetric connection, much higher than expected.  Moreover, the fourth-order term involves the covariant derivative of the curvature only, which is vanishing in locally symmetric space. We show that pole ladder is actually locally exact in these spaces, and even almost surely globally exact in Riemannian symmetric manifolds. These properties make pole ladder a very attractive alternative to other methods in general affine manifolds with a symmetric connection.
\end{abstract}


\begin{figure}[t!]
\begin{center}
	\includegraphics[width=5cm]{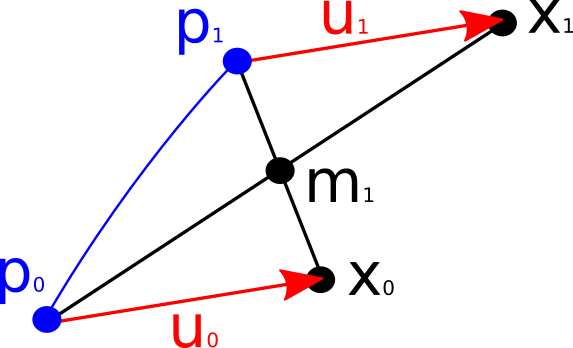}\hspace{2cm}
	\includegraphics[width=5cm]{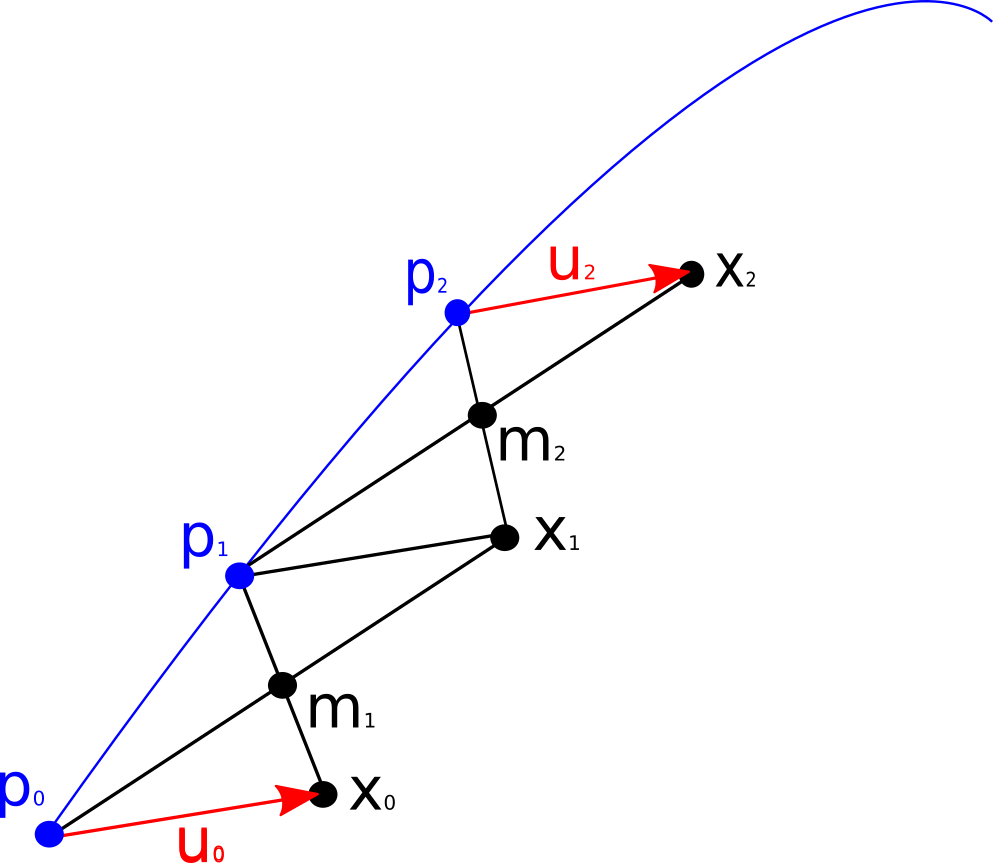}
\end{center}
\caption{Schild's ladder procedure to parallel transports the vector $u_0 = \log_{p_0}(x_0)$ along the sampled curve $(p_0, p_1,\ldots p_n)$. {\bf Left:} First rug of the ladder using an approximate geodesic parallelogram. The vector $u_0 = \log_{p_0}(x_0)$ (rescaled to be small enough if needed) is identified to the geodesic segment $\exp(t u_0)$ for $t\in[0,1]$. One compute the mid-point $m_1$ between $x_0$ and $p_1$ and then extend twice the geodesic from $p_0$ to $m_1$ to obtain the point $x_1 = \exp_{p_0}(2 \log_{p_0}(m_1))$. The vector $u_1 = \log_{p_1}(x_1)$  is a first order approximation of the parallel transport of $u_0$ at $p_1$ \cite{kheyfets_schilds_2000}. {\bf Right:} The method is iterated with rungs at each point sampled along the curve.  Figure adapted from Wikipedia \cite{wiki:SchildsLadder}. 
}\label{SchildsLadder}
\end{figure}

Numerical methods have been proposed in geometric data processing for the parallel transport of vectors along curves in manifolds. The oldest algorithm is probably Schild's ladder, a general method for the parallel transport along arbitrary curves, introduced in the theory of gravitation in \cite{Misner:1973} after Schild's similar constructions \cite{Schild:1970}, and popularized recently by Wikipedia \cite{wiki:SchildsLadder}. The method extends the infinitesimal transport through the construction of geodesic parallelograms (Figure \ref{SchildsLadder}). The method is algorithmically interesting since it only requires the computation of geodesics (initial and boundary value problems) without requiring the knowledge of the second order structure of the space (connection or curvature tensors).  \cite{kheyfets_schilds_2000} proved that the scheme realizes a first order approximation of the parallel transport for a symmetric connection. This makes sense since the skew-symmetric part of the connection, the torsion, does not impact the geodesic equation. Schild's ladder is nowadays increasingly used in non-linear data processing and analysis to implement parallel transport in Riemannian manifolds. One can cite for instance \cite{lorenzi:inria-00616210} for the parallel transport of deformations in computational anatomy or \cite{hauberg_unscented_2013} for parallel transporting the covariance matrix in Kalman filtering. 

In the medical analysis domain, groups of diffeomorphisms are used to encode the shape differences between objects (point sets, curves, surfaces, images), a method coined diffeomorphometry. In this setting, the shape changes measured for one individual need to be transported in a common geometry. A typical example is the analysis of structural brain changes with aging in Alzheimer's disease: the longitudinal morphological changes for a specific subject can be evaluated through the non-linear registration in the geometry of each subject and encoded as the initial tangent vector of a geodesic. For the longitudinal group-wise analysis, these vectors encoding the subject-specific longitudinal trajectories need to be transported in a common reference. To give a geometric structure to diffeomorphisms, right invariant Riemannian kernel metrics are often considered. This is the foundation of the Large Diffeomorphic Deformation Metric Mapping (LDDMM) framework, pioneered by Miller, Trouv\'e and Younes \cite{miller_group_2001}. Within this domain, the use of iterated infinitesimal Jacobi fields for parallel transport was first proposed by  \cite{laurent_younes_jacobi_2007,younes_transport_2008}. A variation  called the fanning scheme  was recently proposed in \cite{Louis_parallel_2017,louis_fanning_hal2017}. A careful numerical analysis showed that the 
scheme is of order one at each step, which provides an approximation inversely proportional to the number of points taken along the curve, similarly to the Schild's ladder.

The symmetric Cartan-Schouten connection provides an alternative affine symmetric space structure on diffeomorphisms where geodesics starting from the identity are the one-parameter subgroups resulting from the flow of Stationary Velocity Fields (SFV).  Geodesics starting from other points are their left and right translation. The idea of parametrizing diffeomorphisms with the flow of SFVs was introduced by \cite{arsigny:inria-00635671}, but it was not fully understood as the geodesics of the Cartan-Schouten connection before \cite{lorenzi:hal-00813835}. It is worth noticing that the canonical affine symmetric space structure on Lie groups provided by the symmetry $s_g(h) = g h\inv g$ (for any elements $g,h$ of the Lie group) got essentially unnoticed in these developments while this is a very important feature of this structure, as we will see in Section \ref{Sec:SymSpaces}.
 In order to implement a parallel transport algorithm that remains consistent with the numerical scheme used to compute the geodesics, \cite{lorenzi:inria-00616210,lorenzi:hal-00813835} proposed to adapt Schild's ladder to image registration with deformations parametrized by SVF. Interestingly, the Schild's ladder implementation appeared to be more stable in practice than the closed-form expression of the symmetric Cartan-Schouten parallel transport on geodesics. The reason is probably the inconsistency of numerical schemes used for the computation of the geodesics and for the transformation Jacobian in the implementation of this exact parallel transport formula. 

Then, we realized that parallel transport along geodesics could exploit an additional symmetry by using the geodesic along which we want to transport as one of the diagonal of the geodesic parallelogram: this gave rise to the pole ladder scheme \cite{lorenzi:hal-00870489}.  Pole ladder combined with an efficient numerical scheme based on the Baker-Campbell-Hausdorff formula was found to be more stable on simulated and real experiments than the other parallel transport scheme tested for the parallel transport in Lie groups. This result and the higher symmetry led us to conjecture that pole ladder could actually be a higher order scheme than Schild's ladder. However, the numerical analysis of the methods remained difficult, especially in affine connection spaces which have no invariant Riemannian metrics.

In this paper, we build on a work of Gavrilov \cite{gavrilov_algebraic_2006} on the Taylor expansion of the composition of two exponentials in  affine connection spaces to establish in Section \ref{Sec:PoleLadder} the approximation provided by one step of the pole ladder up to order 5. It is remarkable that the scheme is of order 3, thus much higher than the first order of the other parallel transport schemes. This makes pole ladder a very attractive alternative to Schild's ladder in general affine manifolds. Moreover, the fourth-order term involves the covariant derivative of the curvature only. 
 Since a vanishing covariant derivative is the characteristic of a locally symmetric space, this suggests that the error could vanish in this case. Thus, we investigate in Section \ref{Sec:SymSpaces} local and global symmetric spaces with affine or Riemannian structure. We show that pole ladder is actually locally exact in one step, and even almost surely globally exact in Riemannian symmetric manifolds. The key feature is that the differential of the symmetry is the negative of the parallel transport, so that the parallel transport along a geodesic segment can be realized using the composition of two symmetries (a transvection). 
Even if this result appears to be new for the geometric computing community, is was already known in mathematics  (the construction is pictured for instance on \cite[Fig.5, p.168]{trofimov_introduction_1994}). The contribution is thus in the connection between the different communities here.


\section{Pole ladder in an affine connection spaces}
\label{Sec:PoleLadder}

This section establishes a Taylor expansion of the approximation provided by one step of pole ladder. We first detail the pole ladder algorithm. We then turn to an equivalent on affine manifolds of the BCH formula for Lie group  before establishing the approximation order of pole ladder. 

\subsection{Pole ladder algorithm}

\begin{figure}[htb!]
\begin{center}
	\includegraphics[width=5cm]{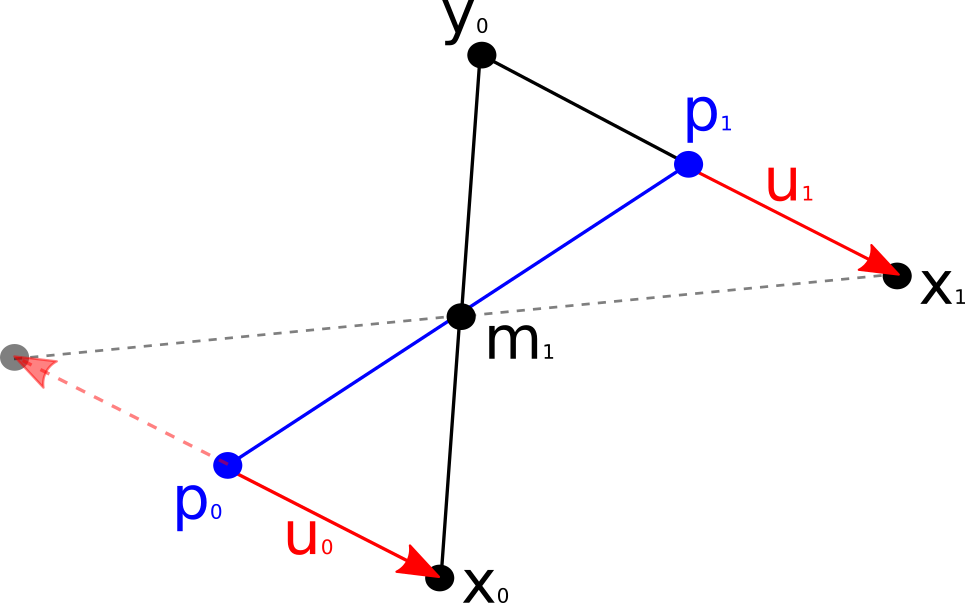}\hspace{2cm}
	\includegraphics[width=4cm]{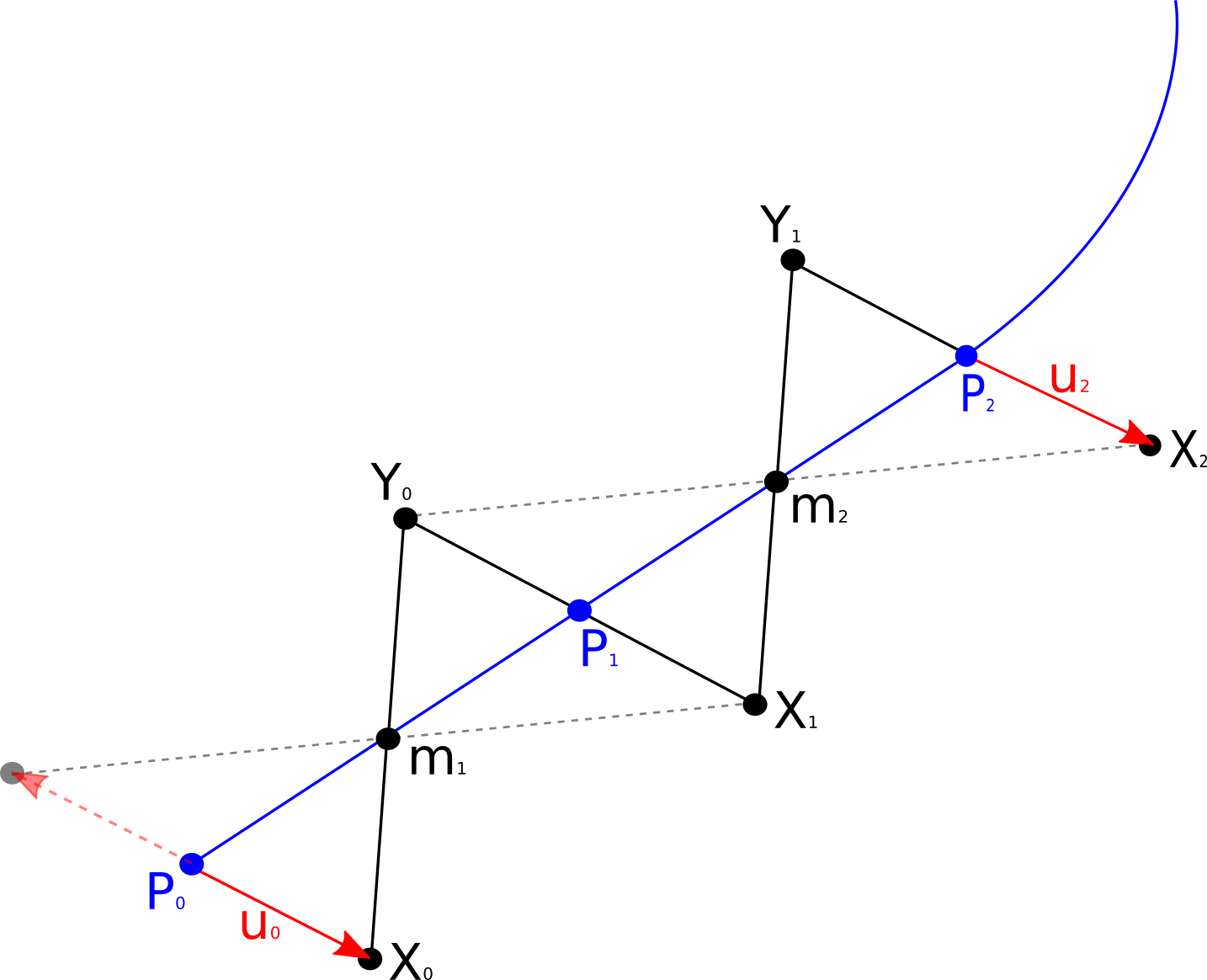}
\end{center}
\caption{Schematic of the pole ladder procedure to parallel transports the vector $u_0=\log_{p_0}(x_0)$ along the geodesic by arc curve $(p_0, p_1,\ldots p_n)$. {\bf Left:} First rug of the ladder using an approximate geodesic parallelogram.  {\bf Right:} The method is iterated with rungs at each point sampled along the curve.  
}\label{PoleLadder}
\end{figure}

Pole ladder is a modification of Schild's ladder for the parallel transport along geodesic  curves which is based on the observation that this geodesic can be taken as one of the diagonals of the geodesic parallelogram. The general procedure is described in Fig.\ref{PoleLadder}: the vector $u_0 = \log_{p_0}(x_0) \in T_{p_0}\M$ (rescaled if needed to be small enough) is identified to the geodesic segment $\exp_{p_0}(t u_0)$ for $t \in [0,1]$. One compute the mid-point $m_1$ between $p_0$ and $p_1$ and then extend twice the geodesic from $x_0$ to $m_1$ to obtain the point $y_1 = \exp_{x_0}(2 \log_{x_0}(m_1))$. The vector $u_1 = -\log_{p_1}(y_1) \in T_{p_1}\M$  is the pole ladder approximation of the parallel transport of $u_0$ at $p_1$. An alternative is to first compute $x'_0 = \exp_{p_0}(-u_0)$ and then extend twice the geodesic from $x'_0$ to $m_1$ to obtain the point $y'_1 = \exp_{x'_0}(2 \log_{x'_0}(m_1))$. The vector $u'_1 =\log_{p_1}(y'_1)$  is another approximation of the parallel transport of $u_0$ at $p_1$.

\subsection{Algorithms for one step of the ladder}

 Pole ladder assumes that the successive points of the geodesic curve are in a sufficiently small convex normal neighborhood so that this mid-point is unique. In order to analyze one single step of the ladder along a geodesic segment, we simplify the notations in the sequel according to figure \ref{PoleLadderNotations}. The original pole ladder procedure is summarized in Algorithm \ref{Algo:Pole1}.

\begin{figure}[htb!]
\begin{center}
	\includegraphics[width=10cm]{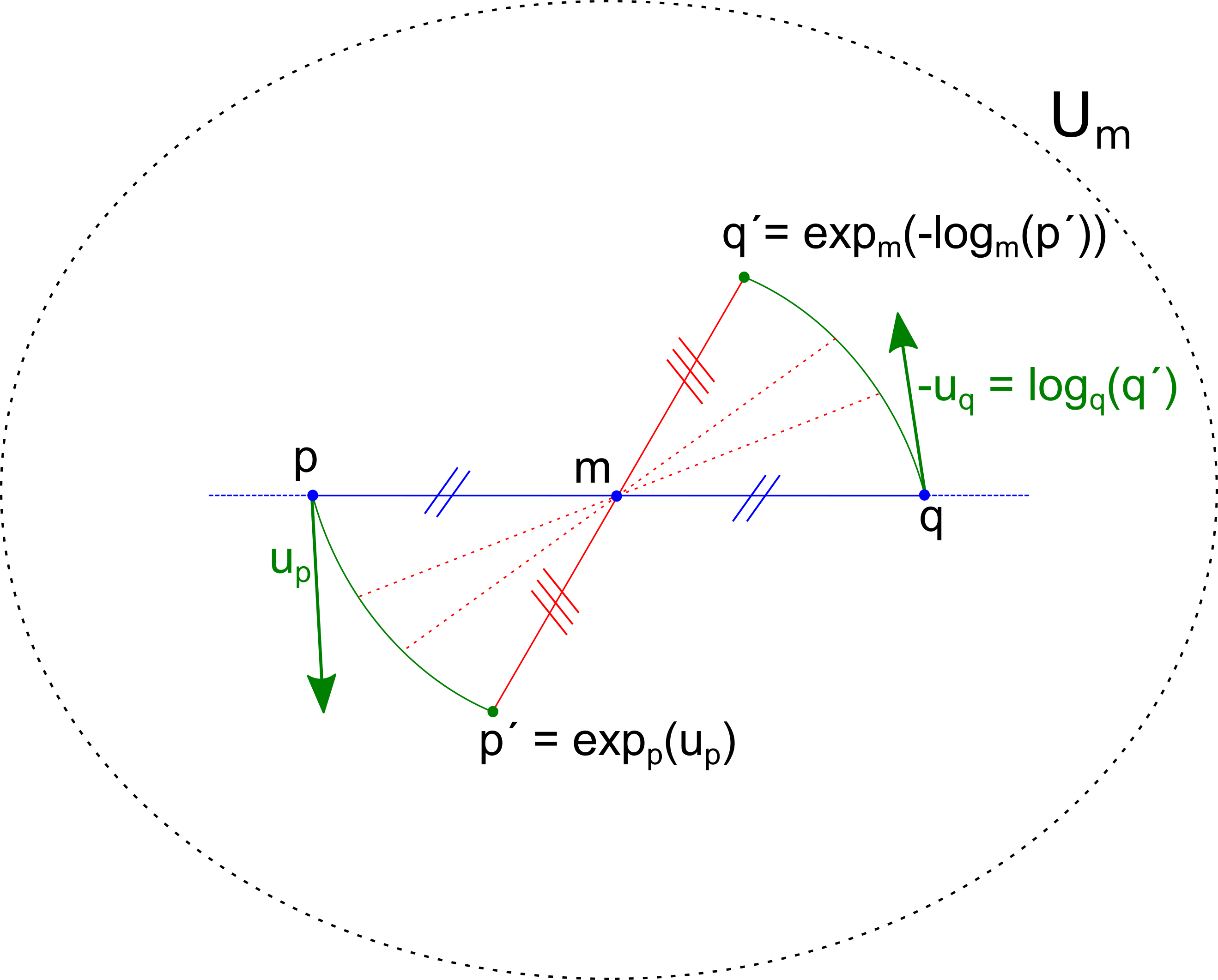}
\end{center}
\caption{One step of the pole ladder scheme.  
}\label{PoleLadderNotations}
\end{figure}

\begin{algorithm}[htb!]
\caption{Parallel transport of vector $u_p \in T_p\M$ along the geodesic segment $[p,q]$.\label{Algo:Pole1}} 
\begin{algorithmic}[1]
\State Compute the midpoint $m = \exp_p(\textstyle \frac{1}{2}\log_p(q))$ between $p$ and $q$. 
\State Compute the end-point $p'= \exp_p(u_p)$ of the geodesic segment $\gamma(t)=  \exp_p(t u_p)$.
\State Compute the geodesic from $p'$ to $m$ and shoot twice to get $q'= \exp_{p'}(2 \log_{p'}(m))$
\State Return the vector $u_q = -\log_q{q'}$  
\end{algorithmic}
\end{algorithm}

Instead of encoding a geodesic $\gamma(t)=  \exp_p(t u_p)$  by its initial point $p$ and initial tangent vector $v$, we can also encode a geodesic segment by its end-points $[p,p'=\exp_p(u_p)]$. We can also rephrase the doubling of the geodesics with a mid-point formulation. This leads to an alternative version of pole ladder using mid-point geodesic symmetries (Algorithm \ref{Algo:Pole2}). Although the two versions of pole ladder are theoretically completely equivalent,  it was experimentally observed that this second version was numerically much more stable when applied to diffeomorphisms acting on images \cite{Jia:MICCAI:2018}. This is most probably be due to the more symmetric formulation that  minimizes numerical errors.

\begin{algorithm}[htb!]
\caption{Parallel transport of the geodesic segment $[p,p']$ along the geodesic segment  $[p,q]$ using mid-point symmetries. \label{Algo:Pole2}} 
\begin{algorithmic}[1]
\State Compute the midpoint $m = \gamma_{[p,q]}(1/2)$ of $[p,q]$.
\State Compute the point $q'= \exp_m(-\log_m(p'))$ symmetric to $p'$ with respect to $m$, i.e such that $m$ is the mid-point of the geodesic segment $[p',q']$.
\State Compute the point $q''= \exp_q(-\log_q(q'))$ symmetric to $q'$ with respect to $q$, i.e such that $q$ is the mid-point of the geodesic segment $[q',q'']$.
\State Return the geodesic segment $[q, q'']$.  
\end{algorithmic}
\end{algorithm}

Notice that the mid-point $m$ of a geodesic segment $[p,q]$ is well defined even in affine connection spaces: this is the point of coordinate $1/2$ in arc-length parametrization. This is in particular an exponential barycenter: $\exp_m(p) + \exp_m(q) =0$. In a Riemannian manifold, it would also be a minimizer of the square distance to the two points, necessarily a Fr\'echet mean if we assumed that all our points belong to a sufficiently small convex neighborhood.


\subsection{A BCH-type formula on affine connection spaces}

In order to analyze the pole ladder approximation of parallel transport in a local coordinate system, we need to compute the Taylor expansion of the parallel transport along a geodesic segment and of the composition of two Riemannian exponentials. Low order Taylor expansion of the Riemannian metric and of  geodesic equations are traditional in Riemannian geometry since Gauss to establish the existence of normal coordinate system,  Gauss lemma or the infinitesimal change of the volume of a geodesic ball due to the curvature. This is also what is used for the proof of \cite{kheyfets_schilds_2000} for Schild's ladder. However,  establishing expansions above the order 3  is much more difficult, and it turns out that we need even more for the pole ladder. One solution was provided by \cite{brewin_riemann_1997,brewin_riemann_2009} with the symbolic computer algebra system Cadabra. In this setting, the starting point is the (pseudo)-Riemannian metric, and it is not so easy to prove that the formulas continue to hold for affine connection spaces.  More recently, a simpler algebraic formulation of higher order covariant derivatives in affine connection spaces \cite{gavrilov_algebraic_2006} led Gavrilov to obtain a combined Taylor expansion of the parallel transport and the composition of exponentials \cite{gavrilov_double_2007}. After recalling this result in our notations, we adapt this formula to the analysis of pole ladder below.

We consider a small convex normal neighborhood $U_m$ of a point $m\in \M$ in an affine connection space $(M, \nabla)$. In a Riemannian manifold, we can chose a sufficiently small regular geodesic ball $B(m,r)$ (see  \cite{karcher77,kendall90}). This means that any two points $x, y$ of the neighborhood can be can be joined by a unique geodesic segment $[x,y]$ and we can define the parallel transport $u_y = \Pi_x^y u \in T_y\M$ of a tangent vector $u \in T_x\M$ along this geodesic. The midpoint of each geodesic segment is also well defined in this neighborhood. 
%
%
The double exponential is defined as the mapping $T_x\M \times T_x\M \rightarrow \M$:
\[
\exp_x(v,u) = \exp_y( \Pi_x^y u) \quad \text{with} \quad y = \exp_x(v).
\]
For sufficiently small vectors $u,v$, one can define the log of this expression (see Figure \ref{Fig:BCH} for notations), which is a formal analog of the Baker-Campbell-Hausdorff (BCH) formula for Lie groups.

\begin{figure}[htb!]
\centering
	\includegraphics[width=8cm]{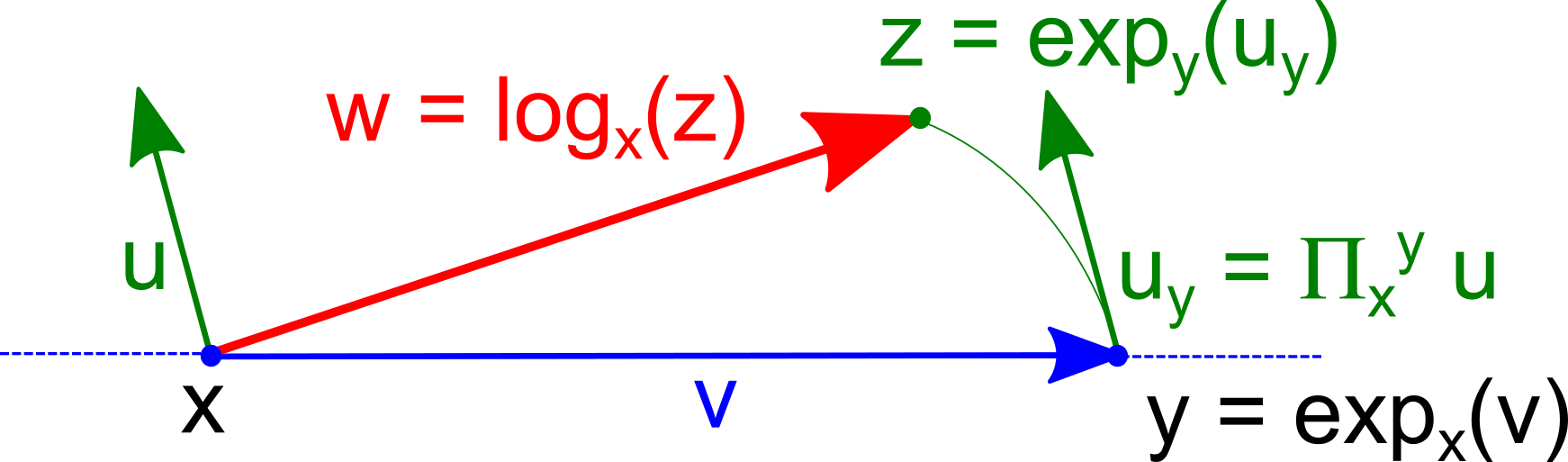}
\caption{The log of the composition of two exponentials (BCH-type formula) in a normal coordinate system at x.}\label{Fig:BCH}
\end{figure}

\begin{theorem}[Gavrilov's BCH-type formula on manifolds \cite{gavrilov_algebraic_2006,gavrilov_double_2007}]\label{thm:BCH}
Let $(M,\nabla)$ be an affine connection space with vanishing torsion and Riemannian curvature tensor $R(u,v)$. In a sufficiently small neighborhood of $0 \in T_x\M \times T_x\M$, the log of the double exponential:
\begin{equation}
h_x(v,u) = \log_x(\exp(x)(v,u) )= \log_x(  \exp_{\exp_x(v)}( \Pi_x^{\exp_x(v)} u)   )
\label{eq:BCH}
\end{equation}
has the following series expansion:
\begin{equation}
\begin{split}
h_x(v,u) & =  v+u + \frac{1}{6} R(u,v)v + \frac{1}{3} R(u,v) u 
           \\ &    + \frac{1}{12}\nabla_v R(u,v)v  + \frac{1}{24}\nabla_u R(u,v)v 
					\\ &   + \frac{5}{24}\nabla_v R(u,v)u + \frac{1}{12}\nabla_u R(u,v)u + O(\|u+w\|^5).
\end{split}
\end{equation}
If the connection is not symmetric, then the torsion $T(u,v) = \nabla_u v - \nabla_v u -[u,v]$ appears as a second order term and in all higher order terms.
\end{theorem}
In the above equations, $O(\|u+w\|^5)$ denotes homogeneous terms of degree 5 and higher in $u$ and $v$ (the norm of $T_x\M$ can be taken arbitrarily), and the curvature tensor value is taken at $x$.  The full series expansion with torsion up to order 5 is provided in  \cite{gavrilov_algebraic_2006}.

\subsection{Taylor expansion of pole ladder}
\begin{figure}[bt!]
\begin{center}
	\includegraphics[width=10cm]{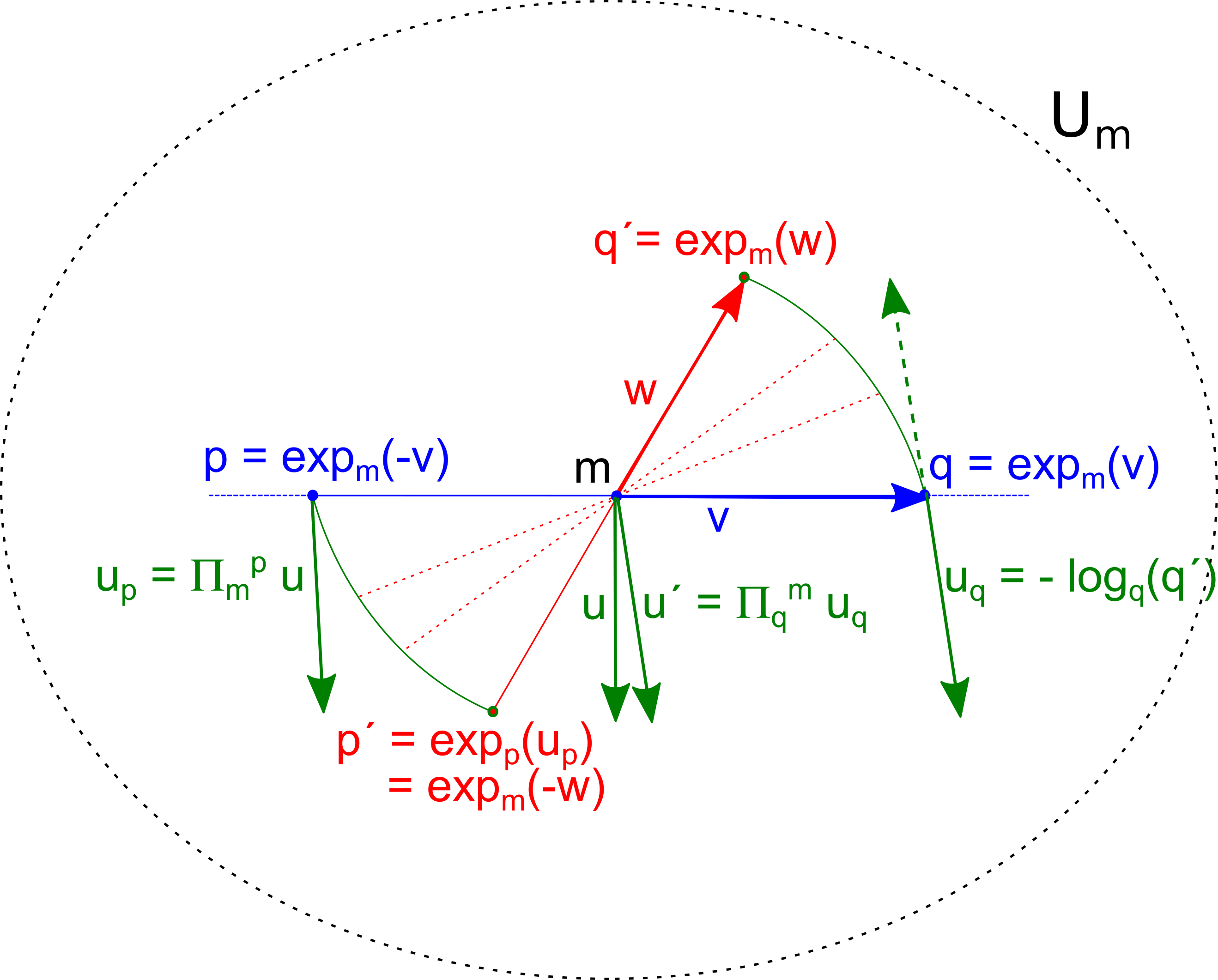}
\end{center}
\caption{Analysis of one pole ladder step in a normal coordinate system at $m$.}\label{Fig:PoleLadderAnalysis}
\end{figure}

We now have the tools to study the second version of pole ladder (algorithm \ref{Algo:Pole2}).
We consider that all points belong to a sufficiently small convex normal coordinate system $U_m$ centered at $m$ with the conventions of Figure \ref{Fig:PoleLadderAnalysis}. The geodesic $\exp_m( t v)$ starting at $m$ with tangent vector $v$ is a straight line in this chart going from point $p$ at time -1 to point $q$ at time 1. We denote by $u= \Pi_p^m u_p$ the parallel transport of the vector $u_p$ along the geodesic segment $[p,m]$. This is equivalent to $u_p = \Pi_m^p u$. We denote by $-w$ the coordinates of $p'= \exp_p( u_p)$ in the normal coordinate system at $m$: $-w = \log_m( \exp_p( \Pi_m^p u ))$ with $p=\exp_m(-v)$. Using the previous BCH formula, we have:
\begin{equation}
\label{eq:w_pole1}
\begin{split}
w  =  & - h_m( -v,u) =  v - u  - \frac{1}{6} R(u,v)v + \frac{1}{3} R(u,v) u 
             + \frac{1}{12}\nabla_v R(u,v)v \\ &   - \frac{1}{24}\nabla_u R(u,v)v 
					   - \frac{5}{24}\nabla_v R(u,v)u + \frac{1}{12}\nabla_u R(u,v)u + O(\|u+w\|^5).
\end{split}
\end{equation}
Notice that the geodesic from $p$ to $p'$  deviates from the straight line in the normal coordinate system at $m$ chart  because of curvature.

We have an equivalent construction for the symmetric part of the structure: the symmetric of $p'$ with respect to $m$ is $q'=\exp_m(w)$. Taking the log at $q$, we get the opposite of the pole ladder transport  $- u_q = \log_q( q')$. In order to compare this vector to the exact parallel transport in a symmetric way, we now parallel transport it along the geodesic segment $[q,m]$  to get  $u'= \Pi_q^m u_q$, that has to be compared with $u$. Rephrasing this combination of transformations in the reverse way, we get that $u'$ is the solution of $w = \log_m( \exp_q( \Pi_m^q (-u') ))$ with $q=\exp_m(v)$. Using Gavrilov's formula, we can write the following polynomial expansion of $w=h_m(v,-u')$ in the variables $v$ and $u'$:
\[
\begin{split}
w  =  & h_m(v,-u') =  v -u' - \frac{1}{6} R(u',v)v + \frac{1}{3} R(u',v) u' 
           \\ &    - \frac{1}{12}\nabla_v R(u',v)v  + \frac{1}{24}\nabla_u' R(u',v)v 
					\\ &   + \frac{5}{24}\nabla_v R(u',v)u' - \frac{1}{12}\nabla_{u'} R(u',v)u' + O(\|u'+w\|^5).
\end{split}
\]
However, what we want to obtain is the polynomial expansion of $u'$ with respect to the variables $v$ and $w$. Such a polynomial can be written:
\[
\begin{split}
u' = & A_v + A_w +B_{vv} + B_{vw} + B_{ww} + C_{vvv} + C_{vvw} + C_{vww} +C_{www} 
\\ & + D_{vvvv} +D_{vvvw} + D_{vvww} + D_{vwww} + D_{wwww} + O(\|u'+w\|^5),
\end{split}
\]
where each term is a multi-linear map in the indexed variables. Plugging this expression into Gavrilov's formula above, we can identify each term iteratively for increasing orders: the first order is trivially 
$u'= v - w + O(\|v+w\|^2)$ and the second order terms have to vanish because there are none in the BCH-type formula. Introducing the third order $u' = v - w +  C_{vvw} + C_{vww} +C_{www} + O(\|v+w\|^4)$ into Gavrilov's formula, we get:
\[ \textstyle
  C_{vvw} + C_{vww} +C_{www}    =   - \frac{1}{6} R(v - w,v)v + \frac{1}{3} R(v - w,v) (v - w)
         + O(\|v+w\|^4).
\]
Simplifying the expression thanks to the skew symmetries, we obtain:
\[
\begin{split}
u'= &  v - w   - \frac{1}{6} R(w,v) v + \frac{1}{3} R(w,v) w + O(\|v+w\|^4).
\end{split}
\]
The forth order is obtained by:
\[
\begin{split}
D_{vvvv}  & +D_{vvvw} + D_{vvww} + D_{vwww} + D_{wwww}   =  
           \\ &    - \frac{1}{12}\nabla_v R(v - w,v)v  + \frac{1}{24}\nabla_{(v - w)} R(v - w,v)v 
					\\ &   + \frac{5}{24}\nabla_v R(v - w,v) (v - w) - \frac{1}{12}\nabla_{(v - w)} R(v - w,v)(v - w) + O(\|v+w\|^5),
\end{split}
\]
which gives after simplification:
\[
 \begin{split}
u'= &  v - w   - \frac{1}{6} R(w,v) v + \frac{1}{3} R(w,v) w 
       - \frac{1}{12}\nabla_v R(w,v) v  - \frac{1}{24}\nabla_{w} R(w,v)v 
         \\ &  + \frac{1}{8}\nabla_v R(w,v) w + \frac{1}{12}\nabla_{w} R(w,v)w
					+ O(\|v+w\|^5).
\end{split}
\label{eq:w_pole2}
\]

Finally, we can substituting the value of $w$ from Equation (\ref{eq:w_pole1}) into this expression.
This leads to: 
\begin{equation}
 \begin{split}
u'= &  u      + \frac{5}{12}\nabla_v R(u,v)u - \frac{1}{6}\nabla_v R(u,v) v 
     \\ &        + \frac{1}{12}\nabla_{u} R(u,v)v  - \frac{1}{6}\nabla_u R(u,v)u  
					+ O(\|v+u\|^5).
\end{split}
\label{eq:PoleApprox}
\end{equation}

\begin{theorem}[Pole ladder is a third order scheme for parallel transport]
\label{thm:PoleApprox}
In a sufficiently small convex normal neighborhood of a point in an affine connection space $(M, \nabla)$ with symmetric connection, the error on one step of pole ladder to transport the vector $u$ along a geodesic segment of tangent vector $[-v,v]$  (all quantities being parallel translated at the mid-point, see Figure \ref{Fig:PoleLadderAnalysis}) is:
\[
u'-u = \frac{1}{12}\left( \nabla_v R(u,v)(5 u - 2v) + \nabla_{u} R(u,v)(v -2 u) \right)
					+ O(\|v+u\|^5).
\] 
\end{theorem}

Instead of performing first a symmetry at $m$ and then at $q$, an alternative version of pole ladder is to first perform a symmetry at $p$ and then at $m$. One could think of averaging the two versions of the pole ladder to get a more accurate transport. The error on this alternative version is: 
\[
u'' - u = - \frac{1}{12}\left( \nabla_v R(u,v)( 5 u  +v) + \nabla_{u} R(u,v) (v +2u) \right)
					+ O(\|v+u\|^5).
\] 
Unfortunately, we see that averaging does not kill all the third order terms.


\section{Pole ladder in symmetric spaces}
\label{Sec:SymSpaces}

Theorem \ref{thm:PoleApprox} shows that the main source of error in the pole ladder procedure is not the curvature itself but its covariant derivative. Since a vanishing covariant derivative is the characteristic of a locally symmetric space, this suggests that the error could vanish also in higher order terms. Thus, we investigate in this section local and global symmetric spaces with affine or Riemannian structure. The description of these spaces closely follows the definitions of \cite{Postnikov:2001}[Chapter 4]. The interested read should also refer to \cite{helgason_differential_1962}.
We show that pole ladder is actually locally exact in one step, and even almost surely globally exact in Riemannian symmetric manifolds. The key feature is that the differential of the symmetry is the negative of the parallel transport, so that the parallel transport along a geodesic segment can be realized using the composition of two symmetries, which is called a transvection. 
Even if this result was not really known in the geometric computing community, the construction was already  pictured for instance in \cite[Fig.5, p.168]{trofimov_introduction_1994}). 

\subsection{Locally  symmetric affine connection spaces}

 An affine connection space is said to be locally symmetric if the connection is torsion free and if the the curvature tensor is covariantly constant ($\nabla R=0$). This second condition is equivalent to the preservation of the curvature tensor by the parallel transport along any path. When the manifold is endowed with a (pseudo-) Riemannian metric, one says that it is a locally symmetric (pseudo-) Riemannian space if the covariant derivative of its curvature tensor with respect to the Levi-Civita connection vanishes identically.

Locally symmetric affine connection spaces can be characterized by the geodesic symmetry $s_m$  which maps any point $p$ of a normal neighborhood $U_m$ of $m$ to the reverse point on the geodesic
$q = s_m(p) = \exp_m(-\log_m(p))$. The differential of this mapping is clearly a central symmetry on (a symmetric neighborhood of 0 in) the tangent space at $m$: $(Ds_m)_m = -\Id$, and $m$ can thus be seen as a midpoint of the geodesic segment $[p,q]$ using the exponential barycenter definition since $\log_m(p) + \log_m(q) =0$. 

Using this geodesic symmetry, pole ladder can be rewritten as the composition of a symmetry at $m$ followed by a symmetry at $q$ (or alternatively as a symmetry at $p$ followed by a symmetry at $m$). Such a composition of two symmetries preserves the orientation and is called a transvection (or simply a translation). Moreover, its differential realizes the parallel transport:
\begin{theorem}[Pole ladder is locally exact in locally symmetric spaces] 
In a locally  symmetric affine connection space with geodesic symmetry $s_m(p)$ at point $m \in \M$, the pole ladder transport along the geodesic segment $[p, q=s_m(p)]$ of the tangent vector $u_p \in T_p\M$ (identified to the geodesic segment $[p, p'=\exp_p(t u_p)]$)   is  the tangent vector $u_q = \log_q(q') \in T_q\M$ (identified to the geodesic segment $[q, q'= s_q(s_m(p'))]$). Alternatively, one can start by the symmetry at $p$ to compute $q'= s_m(s_p(p'))$. For a sufficiently small time $t$, both procedure realize exactly the parallel transport. 
\end{theorem}

%
%
%

\begin{proof}
We consider two points $p$ and $q$ that are sufficiently close in a symmetric normal neighborhood $U_m$ of their midpoint $m$. The geodesic $\exp_p(tu_p)$ starting at $p$ with tangent vector $u_p$ remains in $U_m$ for $t$ small enough. The geodesic symmetry $s_m$ is an affine mapping in $U_m$ that maps geodesics to geodesics \cite{Postnikov:2001}[Chapter 3]. Thus, it maps $\exp_p(tu_p)$ to a geodesic $\exp_q(t u_q ) = s_m(\exp_p(tu_p)) = \exp_{s_m(p)}( t (Ds_m)_p u_p)$ starting at $q = s_m(p)$ with tangent vector $u_q = (Ds_m)_p u_p$. Now, the key property is that the differential of the central symmetry $s_m$ at $p$ is is a linear maps from $T_p\M$ to $T_{s_m(p)}\M$ which realizes exactly the (negative of) the parallel transport along the geodesic segment $[p, s_m(p)]$:  $(Ds_m)_p = -\Pi_{\gamma}$. This property is described in \cite[Eq.(6), p.164]{helgason_differential_1962} and in the proof of Proposition 4.3 of \cite[p.53-54]{Postnikov:2001}. The second symmetry is well defined for a sufficiently small time $t$ and reverses the sign.
\end{proof}

Because we only assumed so far a locally symmetric affine connection, the size of the normal neighborhood containing $p$ and $q$ and the maximal time $t$ cannot be specified without an auxiliary metric. More traditional symmetric (pseudo)-Riemannian spaces provide additional tools to extend this result.

\subsection{Riemannian symmetric spaces}

A connected affine connection space $\M$ is a (globally) symmetric space if there exists at each point a smooth involution $(p,q)\in \M \times \M \rightarrow s_p(q) \in \M$ for which $p$ is an isolated point (i.e. a symmetry)  and which is stable by composition: for any two points $p,q\in\M$,  $s_q \circ s_p \circ s_q = s_{s_p(q)}$. 
A symmetric space is geodesically complete and is obviously locally symmetric. The converse holds modulo discrete quotients: in essence, for each connected and geodesically complete locally symmetric space, there exists a symmetric covering space, which can be taken to be simply connected.
Given a symmetric space, an even number of composition of symmetries generate a translation (or transvection) on the manifold. The Lie group of translations $\G$ is acting transitively on $\M$, and its quotient space $\G/\HH$ by the isotropy subgroup $\HH$ of one point of the manifold can be identified with the symmetric space itself. 

Conversely, Cartan theorem shows that  a homogeneous space with an involutive automorphism $\sigma: \G \rightarrow \G$ satisfying specific properties detailed below is a symmetric space. Because $\sigma$ is an involution, its differential $D_{\Id}\sigma$ has eigenvalues $\pm 1$ which decompose the Lie algebra into the direct sum $\mathfrak g = \mathfrak h \oplus \mathfrak m$: the +1-eigenspace $\mathfrak {h}$ is the Lie algebra of $\HH$, and the $-1$-eigenspace $\mathfrak {m}$ is an ${\mathfrak {h}}$-invariant complement to ${\mathfrak {h}}$ in ${\mathfrak {g}}$ that can be identified to the tangent space at $\M$. These two conditions read in the Lie theoretic characterization $[\mathfrak h, \mathfrak h] \subset \mathfrak h$ and $[\mathfrak h, \mathfrak m] \subset \mathfrak m$, which means that a symmetric space is a reductive homogeneous space. However, reductive homogeneous spaces are not all symmetric:  the key feature is that ${\mathfrak {[m,m] \subset h}}$ (the composition of two symmetries on $\M$ has to be a translation on $\mathfrak h$).  

A natural structure on a symmetric space has to be invariant under its symmetry. This condition uniquely define a symmetric connection on each symmetric space.  For instance, a connected Lie group is a symmetric space with the symmetry $s_p(q) = p q\inv p$ for two points of the group $p, q \in \G$. The associated canonical connection is the mean (or symmetric) Cartan-Schouten connection $\nabla_X Y = \frac{1}{2}[X,Y]$ which was used for defining bi-invariant means on Lie groups even in the absence of an invariant metric \cite{pennec:hal-00699361}. On a more general symmetric space $\M \simeq \G/\HH$, the canonical connection is induced by the symmetric Cartan-Schouten connection on the Lie group of translations $G$, and geodesics of the symmetric space with this canonical connection are realized by the action of one-parameter subgroups of $G$ on a point of the manifold. 

When $\HH$ is compact, one can moreover obtain an invariant metric on $\M \simeq \G/ \HH$ (simply choose a metric on the Lie algebra $\mathfrak g$  and average it over $\HH$). We obtain in this case a Riemannian symmetric manifold where the symmetry is now an isometry. The  Levi-Civita connection of this invariant Riemannian metric coincides with the canonical connection. The invariant metric need not be uniquely define up to a scalar factor. For instance, that there is a one-parameter family of $GL(n)$-invariant metrics on SPD matrices \cite{Pennec:HDR:2006}. 

With an invariant Riemannian metric, we can define the maximal injectivity domain of the exponential map: let $U_p \in T_p\M$ be the set of all vectors $v \in T_p\M$ such that $\gamma(t) = \exp_p(tv)$ is a minimizing geodesic for $t \in [0, 1+\epsilon]$ for some $\epsilon > 0$. This is an open star-shape domain containing $0$ limited by the tangential cut locus $\partial U_p$ where geodesics stop to be minimizing. The image of the tangential cut locus by the exponential map is  the cut locus $Cut(p) = \exp_p(\partial U_p)$, which has null measure with respect to the canonical Riemannian measure. The maximal injectivity domain $U_p$ is 
 mapped diffeomorphically by $\exp_p$ to the normal neighborhood $\M \setminus Cut(p) $ of $p$ in the manifold.   The shortest distance from a point to its cut locus is the injection radius $\mbox{inj}(p)$. Since a Riemannian symmetric space is homogeneous, this radius does not depend on the point $p$ and is equal to its infimum over all the points $\mbox{inj}(\M)$. These definitions allows us to characterize the exactness of pole ladder transport in symmetric spaces.

\begin{theorem}[Pole ladder is almost surely exact in Riemannian symmetric manifolds] 
In a Riemannian  symmetric manifold, pole ladder is exact when $q \not \in Cut(p)$ and $p'=\exp_p(u_p) \not \in Cut(p)$. This happens almost surely since the cut locus has null measure. More restrictive metric conditions of exactness are 
$\dist(p,q) < \mbox{inj}(\M)$ and $\dist(p, p') = \| u_p\|_p  < \mbox{inj}(\M)$. 
\end{theorem}
\begin{proof}
The first condition is needed to have a well defined mid-point. The second condition ensures that the geodesic $\exp_p(t u_p)$ does not meet the cut-locus of $p$ for $t\in[0,1]$, so that the initial tangent vector of the geodesic from $p$ to $p'$ remains $u_p$. Thanks to the global symmetry $s_m$, similar conditions hold for $q$ and $q'$, so that $-u_q$ is the exact parallel transport of $u_p$ along the geodesic segment $[p,q]$.
\end{proof}

\bibliographystyle{alpha} 


\newcommand{\etalchar}[1]{$^{#1}$}
\hyphenation{ Chris-to-dou-la-kis Fach-ge-sprach feh-ler-be-hand-lung
  feh-ler-er-ken-nung Han-over Jean-ette Mann-heim Piep-rzyk Reuh-ka-la
  Rus-in-kie-wicz Sa-degh-i-yan Worm-ald zu-griffs-ver-fahr-en }

\end{document}